\documentclass[12pt]{amsart}
\usepackage[margin=1in]{geometry} 

\usepackage{enumitem}
\usepackage{mathrsfs}
\usepackage{mathtools}
\usepackage{amsfonts}
\usepackage{amsmath,amsthm,amscd}
\usepackage{graphicx}
\usepackage{subfigure}
\usepackage[abbrev]{amsrefs}
\usepackage{color}
\usepackage{cancel}
\usepackage{hyperref}
\usepackage[utf8]{inputenc}
\usepackage[T1]{fontenc}
\usepackage[english]{babel}
\usepackage{hyphenat}
\usepackage{wrapfig}
\usepackage{graphicx}
\usepackage{tikz}
\everymath{\displaystyle}
\hypersetup{
    colorlinks=true,
    linkcolor=blue,
    filecolor=magenta,      
    urlcolor=cyan,
}
\usepackage[todonotes]{changes}
\newtheorem{defi}{Definition}[section]
\newtheorem{lemma}[defi]{Lemma}
\newtheorem{prop}[defi]{Proposition}
\newtheorem{thm}[defi]{Theorem}
\newtheorem*{thm*}{Theorem}
\newtheorem{cor}[defi]{Corollary}
\theoremstyle{definition}
\newtheorem{rmk}[defi]{Remark}
\newtheorem{ex}[defi]{Example}
\newcommand{\bottom}{\lambda_{\operatorname{min}}}
\newcommand{\Ric}{\operatorname{Ric}}

\title{Spectrum of the drift Laplacian on Ricci expanders}
\date{\today.}

\author{Helton Leal
\and 
 Matheus Vieira
\and
Detang Zhou
}
\thanks{H.L. was supported by CNPq - Brazil (150839/2024-3)}
\thanks{D. Z. was partially supported by CNPq/Brazil [Grant: 305364/2019-7, 403344/2021-2] and FAPERJ/ Brazil [Grant: E-26/200.386/2023].}
\address{Instituto de Matem\'atica e Estat\'istica, Universidade Federal Fluminense, S\~ao Domingos,
Niter\'oi, RJ 24210-201, Brazil}
\email{helton\_leal@id.uff.br}
\address{Departamento de Matem\'atica, Universidade Federal do Esp\'irito Santo, Vit\'oria, ES, Brazil}
\email{matheus.vieira@ufes.br}
\address{Departamento de Geometria, Instituto de Matem\'atica e Estat\'istica, Universidade Federal Fluminense, S\~ao Domingos,
Niter\'oi, RJ 24210-201, Brazil}
\email{zhoud@id.uff.br}

\begin{document}

\begin{abstract}
In this paper, we study the spectrum of the drift Laplacian on Ricci expanders. We show that the spectrum is discrete when the potential function is proper, and we show that the hypothesis on the properness of the potential function cannot be removed. We also extend previous results concerning the asymptotic behavior of the potential function on Ricci expanders. This allows us to conclude that the drift Laplacian has discrete spectrum on Ricci expanders whose Ricci curvature is bounded below by a suitable constant, possibly negative. Further, we compute all the eigenvalues of the drift Laplacian on rigid expanders and rigid shrinkers. Lastly, we investigate the second eigenvalue of the drift Laplacian on rigid Ricci expanders whose Einstein factor is a closed hyperbolic Riemann surface.
\end{abstract}

\maketitle

\section{Introduction}

The study of the spectrum of the Laplacian on  Riemannian manifolds is one of the most important topics in geometry. When the spectrum of the Laplacian is discrete, we can write the eigenvalues in increasing order as
$$0 \leq \lambda_1(\Delta) < \lambda_2(\Delta) \leq \lambda_3(\Delta) \leq \cdots\rightarrow\infty.$$
Note that we allow $\lambda_1 (\Delta)$ to be zero, which is different from the usual convention. In particular, for compact manifolds we always have $\lambda_1 (\Delta) = 0$. It is known that on compact manifolds the spectrum of the Laplacian $\Delta$ is discrete, and the classical Lichnerowicz-Obata theorem \cites{L,O} tell us that if the Ricci curvature of a compact manifold $(M^n,g)$ satisfies $\operatorname{Ric} \geq \rho g$ for some constant $\rho>0$, then the smallest nonzero eigenvalue satisfies $\lambda_2(\Delta)\ge \frac{n}{n-1} \rho$, and equality holds if and only if the manifold $M$ is isometric to a round sphere. A version of the Lichnerowicz-Obata theorem was recently obtained for the drift Laplacian $\Delta_f = \Delta-\nabla_{\nabla f}$. The combined works of Bakry-Émery \cite{BakryEmery}, Morgan \cite{Morgan}, Hein-Naber \cite{HeinNaber} and Cheng-Zhou \cite{ChengZhou2017} tell us that if the Bakry-Émery-Ricci curvature of a complete smooth metric measure space $(M,g,f)$ satisfies $\Ric+\nabla^2f \geq \rho g$ for some constant $\rho>0$, then the spectrum of the drift Laplacian $\Delta_f$ is discrete, the smallest nonzero eigenvalue satisfies $\lambda_2 (\Delta_f) \ge \rho$, and equality holds if and only if the manifold $M$ is isometric to a cylinder $N^{k}\times\mathbb{R}^{n-k}$. In case of equality, the multiplicity of $\lambda_2 (\Delta_f)$ is precisely the dimension of the Euclidean factor of the cylinder.

Ricci solitons are very important and  interesting models in mathematics and physics. On the one hand, Ricci solitons generalize Einstein manifolds, as they correspond to smooth metric measure spaces $(M,g,f)$ with constant Bakry-Émery-Ricci curvature, that is $\operatorname{Ric} +\nabla^2f=\rho g$ for some constant $\rho$. On the other hand, Ricci solitons are self-similar solutions of the Ricci flow, often arising as singularity models. While Ricci shrinkers ($\rho>0$) have been the focus of research for the past two decades, Ricci expanders ($\rho<0$) have received detailed attention only recently. In this paper, we are interested in the spectrum of the drift Laplacian on Ricci expanders.

As we mentioned above, as a special case of the combined works of Bakry-Émery \cite{BakryEmery}, Morgan \cite{Morgan} and Hein-Naber \cite{HeinNaber}, the spectrum of the drift Laplacian on a Ricci shrinker is discrete. We show that the spectrum of the drift Laplacian on a Ricci expander is also discrete when the potential function is proper.

\begin{thm}[Theorem \ref{thm:discrete}] Let $\left(M,g,f\right)$ be a complete Ricci expander. Suppose that the potential function $f$ is proper. Then:

(i) The spectrum of the drift Laplacian $\Delta_{f}$ in the space $L_{f}^{2}$ is discrete.

(ii) Assuming that the scalar curvature $S$ is bounded above by a constant, the spectrum of the drift Laplacian $\Delta_{-f}$ in the space $L_{-f}^{2}$ is also discrete.
\end{thm}

The assumption that the potential function $f$ is proper cannot be removed from this theorem. In fact, for the Ricci expander $M^n = \mathbb{H}^k \times \mathbb{R}^{n-k}$ the spectrum of the drift Laplacian $\Delta_f^{M}$ is not discrete (see Example \ref{ex:hyperbolic}).

We would like to point out that the potential function of a Ricci shrinker is always proper. In fact, Cao-Zhou \cite{CaoZhou2009} showed that the potential function $f$ of a Ricci shrinker $(M,g,f)$ satisfies
$$(\rho/2) (r(x)-c_1)^2 \leq f(x) \leq (\rho/2) (r(x)+c_2)^2.$$
For a Ricci expander $(M,g,f)$ with nonnegative Ricci curvature, similar estimates for the function $-f$ were proved by Cao-Catino-Chen-Mantegazza \cite{Caoetal} and Chen-Deruelle \cite{ChenDer}. Using an estimate due to Munteanu-Wang \cite{MunteanuWang}, we can improve the results in \cite{Caoetal} and \cite{ChenDer} by allowing the Ricci curvature to be bounded below by a constant $\eta$, possibly nonpositive. In particular, we obtain a sufficient condition for the potential function of a Ricci expander to be proper.

\begin{thm}[Theorem \ref{thm:asymptotic}]
Let $\left(M,g,f\right)$ be a complete noncompact Ricci expander satisfying \eqref{eq:soliton} with Ricci curvature bounded below by a constant $\eta$ (that is $\operatorname{Ric\geq\eta g}$). Then there exist positive constants $c_1$ and $c_2$ such that
    $$ (\eta-\rho/2)r(x)^2-c_1r(x)\leq -f(x) \leq (-\rho/2)(r(x)+c_2)^2, $$
where $r$ is the distance function.
\end{thm}

Combining the two previous theorems, we obtain the following.

\begin{cor}[Corollary \ref{cor:ricci}]
Let $\left(M,g,f\right)$ be a complete noncompact Ricci expander satisfying \eqref{eq:soliton} and with Ricci curvature bounded below by a constant $\eta>\rho/2$ (that is $\operatorname{Ric\geq\eta g}$). Then the potential function $f$ is proper and the spectrum of the drift Laplacian $\Delta_{f}$ in the space $L_{f}^{2}$ is discrete. Also, assuming that the scalar curvature is bounded above by a constant, the spectrum of the drift Laplacian $\Delta_{-f}$ in the space $L_{-f}^{2}$ is discrete.
\end{cor}

In particular we can apply this corollary to the Ricci expanders constructed by Bryant \cite{Bryant} and Cao \cite{Cao} (see Example \ref{ex:bryantcao}).

The product of a compact Einstein manifold with the Euclidean space, up to a quotient and with a natural potential function, is called a \emph{rigid} soliton (see Petersen-Wylie \cite{PetersenWylie}). These form an important class of Ricci solitons. Huai-Dong Cao conjectured that Ricci shrinkers with constant scalar curvature must be rigid, and this conjecture has been confirmed up to dimension four (see Cheng-Zhou \cite{CZ2023}).  In this paper, we explicitly compute the spectrum of the drift Laplacian on rigid Ricci expanders. We note that a similar result holds in the shrinking case (Theorem \ref{thm:rigid-shrinker}).

\begin{thm}[Theorem \ref{thm:rigid}]
Let $N^k$ be a compact Einstein manifold with Ricci curvature $\operatorname{Ric}_N = \rho g_N$, where $\rho<0$. Consider the triple
$$(M^{n}=N^{k}\times \mathbb{R}^{n-k}, g_M = g_N + g_{\mathbb{R}^{n-k}}, f (x,y) = \left(\rho/2\right)\left|y\right|^{2}).$$
Then:

(i) The triple $\left(M,g_M,f\right)$ is a Ricci expander with $\operatorname{Ric}_M + (\nabla ^2 f)_M = \rho g_M$.

(ii) The spectrum of the drift Laplacian $\Delta_{f}^M$ in the space $L_{f}^{2}\left(M\right)$ is discrete.

(iii) The eigenvalues $\lambda$ of the drift Laplacian $\Delta_{f}^M$  with corresponding eigenfunctions
$u\left(x,y\right)$  are of
the form
\[
\lambda=\mu-\rho\left(n-k+\sum_{i=1}^{n-k}p_{i}\right),\,\,\,\,\,u\left(x,y\right)=v\left(x\right)e^{\left(\rho/2\right)\left|y\right|^{2}}\prod_{i=1}^{n-k}H_{p_{i}}\left(\sqrt{-\rho}y_{i}\right),
\]
where the number $\mu$ is an eigenvalue of the Laplacian $\Delta^{N}$ of the Einstein manifold with corresponding eigenfunction $v\left(x\right)$, the numbers $p_1$, $\dots$, $p_{n-k}$ are nonnegative integers and the function $H_{p_{i}}$ is the Hermite polynomial of degree $p_{i}$.
\end{thm}

Thus, on rigid Ricci shrinkers the first eigenvalue of the drift Laplacian is equal to zero, and, on the other hand, on rigid Ricci expanders the first eigenvalue of the drift Laplacian is nonzero and equal to $- \rho(n-k)$ (of course the multiplicity is $1$ in both cases). Also, on rigid Ricci shrinkers the second eigenvalue of the drift Laplacian is equal to the constant of Ricci the soliton equation with multiplicity equal to the dimension of Euclidean factor, and, on the other hand, on rigid Ricci expanders the second eigenvalue of the drift Laplacian and its multiplicity are both less clear, so a splitting result in the spirit of Cheng-Zhou \cite{ChengZhou2017} may not be true. This led us to investigate the second eigenvalue of the drift Laplacian on the simplest non-trivial example of a rigid Ricci expander we could think of, namely the product of a closed hyperbolic surface with the Euclidean space.

Using Theorem \ref{thm:rigid} and a recent inequality of Karpukhin-Vinokurov \cite{KarpukhinVinokurov}, which improves the classical Yang-Yau inequality \cite{YangYau}, we get:

\begin{thm}[Theorem \ref{thm:second}]
Let $\Sigma$ be a closed Riemann surface with constant curvature $\rho<0$. Let $M^n=\Sigma\times\mathbb{R}^{n-2}$ be the Ricci expander with the product metric and the potential function $f(x,y)=(\rho/2)|y|^2$. Then the spectrum of the drift Laplacian $\Delta_f^M$ is discrete, the first eigenvalue is
$$\lambda_1(\Delta_f^M)=-\rho(n-k),\,\,with\,\,multiplicity\,\,1,$$
and the second eigenvalue is
$$
\lambda_2 (\Delta_f^M) = 
\begin{cases}
-\rho (n-2),\,\,with\,\,multiplicity\,\,n-2, & if\,\,\lambda_2 (\Delta^\Sigma )>-\rho,\\
-\rho (n-2),\,\,with\,\,the\,\,multiplicity\,\,of\,\,\lambda_2 (\Delta^\Sigma)+n-2, & if\,\,\lambda_2 ( \Delta^\Sigma )=-\rho,\\
\lambda_2 ( \Delta^\Sigma) -\rho (n-2),\,\,with\,\,the\,\,multiplicity\,\,of\,\,\lambda_2 ( \Delta^\Sigma ), & if\,\,\lambda_2 ( \Delta^\Sigma )<-\rho.
\end{cases}
$$
Also, if the genus $\gamma\geq 46$, then $\lambda_2 ( \Delta^\Sigma )<-\rho$ and
$$\lambda_2(\Delta_f^M)=\lambda_2(\Delta^\Sigma)-\rho(n-2).$$
\end{thm}

\begin{rmk} Since Ricci expanders are natural generalizations of Einstein manifolds of negative scalar curvature, we can compare their spectrum with the spectrum of hyperbolic manifolds. The spectrum of the Laplacian on hyperbolic surfaces has been a fascinating topic in several mathematical fields including analysis, dynamics, geometry, mathematical physics, number theory and so on for a long time. In the spectral theory of automorphic functions, small eigenvalues (i.e., those lying in the interval [0, 1/4)) of the Laplace operator are of particular interest. We do not even know all the exact values of the first (nonzero) eigenvalue on closed hyperbolic Riemann surfaces of genus $\gamma$.
\end{rmk}

This paper is organized as follows. In Section \ref{sec:prelim}, we introduce the basic notations and definitions. In Section \ref{sec:discrete} we prove Theorem \ref{thm:discrete}, Theorem \ref{thm:asymptotic} and Corolary \ref{cor:ricci}, among other results. In Section \ref{sec:rigid}, we prove Theorem \ref{thm:rigid} and its counterpart for Ricci shrinkers. In Section \ref{sec:second} we discuss the second eigenvalue on rigid Ricci solitons and prove Theorem \ref{thm:second}. In the Appendix (Section \ref{sec:product}), we prove, for the sake of completeness, some elementary results concerning the spectrum of the drift Laplacian on product manifolds.

\section{Preliminaries}\label{sec:prelim}
In this section, we introduce the basic concepts of smooth metric measure spaces, the spectrum of the drift Laplacian and gradient Ricci solitons.

\subsection{Smooth metric measure spaces}
A smooth metric measure space is defined as a triple $(M,g,f)$ of a Riemannian manifold $(M,g)$ and a smooth function $f$ on the manifold $M$. The drift Laplacian $\Delta_f$ is defined as the operator on the space $C^{\infty}$ (the set of smooth functions on the manifold $M$) given by
$$\Delta_f = \Delta - \nabla_{\nabla f}.$$
The space $L^2_f$ is defined as the set of square-integrable functions on the manifold $M$ with respect to measure $e^{-f}dV$, which is a Hilbert space with the inner product
$$\langle u, v \rangle_{L^2_f}= \int_M uve^{-f}.$$
The drift Laplacian $\Delta_f$ is a self-adjoint operator in the space $L^2_f$. That is,  for smooth functions $u$ and $v$ on the manifold $M$ with compact support we have
$$\langle \Delta_f u , v \rangle_{L^2_f} = \langle u , \Delta_f v \rangle_{L^2_f}.$$
The Bakry-Emery-Ricci curvature $\operatorname{Ric}_f$ is defined as the two-tensor on the manifold $M$ given by
$$\operatorname{Ric}_f=\operatorname{Ric}+\nabla^2 f.$$
A smooth metric measure space $(M,g,f)$ is called a gradient Ricci soliton if it has constant constant Bakry-Émery-Ricci curvature, that is
$$\operatorname{Ric}_f = \rho g,$$
for some constant $\rho$.

For products manifolds we indicate the manifold in the operators and norms. For example, the product manifold $M^n = N^k \times \mathbb{R}^{n-k}$ has three natural Laplacians: the Laplacian $\Delta^M$ of the product and the Laplacians $\Delta^N$ and $\Delta^{\mathbb{R}^{n-k}}$ of the factors.

\subsection{Spectrum of the drift Laplacian}
Given a complete smooth metric measure space $(M,g,f)$, the drift Laplacian $\Delta_f$ is a densely defined self-adjoint operator on the Hilbert space $L^2_f$. The spectrum $\sigma(\Delta_f)$ of the drift Laplacian $\Delta_f$ is defined as the set of real numbers $\lambda$ such that the operator $\Delta_f + \lambda I$ does not have a bounded inverse. The spectrum of the drift Laplacian splits into discrete spectrum and essential spectrum. The discrete spectrum $\sigma_d(\Delta_f)$ is defined as the set of real numbers $\lambda$ that are isolated eigenvalues of the drift Laplacian $\Delta_f$ with finite (nonzero) multiplicity. The essential spectrum $\sigma_{ess}(\Delta_f)$ is defined as the complement of the discrete spectrum in the spectrum, that is $\sigma_{ess}(\Delta_f)=\sigma(\Delta_f)\setminus \sigma_d(\Delta_f)$.

The discreteness of the spectrum of the drift Laplacian can be characterized as follows. The spectrum of the drift Laplacian $\Delta_f$ is discrete if and only if the canonical embedding of the space $H^1_f$ into the space $L^2_f$ is compact (see Theorem 10.20 in Grigoryan's book \cite{Grigoryan}). Here the space $H^1_f$ is the set of functions on the manifold $M$ such that the function and its (weak) gradient are in the space $L^2_f$.

When spectrum of the drift Laplacian $\Delta_f$ is discrete, there exists an orthonormal basis $\{\varphi_1, \varphi_2, \varphi_3, \dots \}$ of the space $L^2_f$ formed by eigenfunctions of the drift Laplacian $\Delta_f$ with corresponding eigenvalues
$$0 \leq \lambda_1(\Delta_f) < \lambda_2(\Delta_f) \leq \lambda_3(\Delta_f) \leq \cdots\rightarrow\infty.$$
In this case, the eigenvalues can be calculated as follows. If the eigenvalues $\lambda_1(\Delta_f)$, $\dots$, $\lambda_{k-1}(\Delta_f)$ with corresponding eigenfunctions $\varphi_1$, $\dots$, $\varphi_{k-1}$ are known, then the next eigenvalue $\lambda_k(\Delta_f)$ is given by
$$\lambda_k(\Delta_f)=\inf_u \langle - \Delta_f u , u \rangle_{L^2_f },$$
where the infimum is over the smooth functions $u$ on the manifold $M$ with compact support satisfying $|u|_{L^2_f} = 1$ and $\langle u, \varphi_i \rangle_{L^2_f} = 0$ for $1 \leq i \leq k-1$.

When spectrum of the drift Laplacian $\Delta_f$ is not discrete, motivated by the above characterization, the bottom spectrum $\bottom (\Delta_f)$ of the drift Laplacian $\Delta_f$ is defined as the real number
$$\bottom (\Delta_f)=\inf_u \langle - \Delta_f u , u \rangle_{L^2_f },$$
where the infimum is over the smooth functions $u$ on the manifold $M$ with compact support satisfying $|u|_{L^2_f} = 1$.

When the spectrum  of the drift Laplacian $\Delta_f$ is discrete, we have $\bottom (\Delta_f) = \lambda_1 (\Delta_f)$.

\subsection{Ricci solitons}

A gradient Ricci soliton is defined as a triple $(M,g,f)$ of a Riemannian manifold $(M,g)$ and a smooth function $f$ on the manifold $M$ such that
\begin{equation}\label{eq:soliton}
\operatorname{Ric} +\nabla^2f=\rho g,
\end{equation}
for some constant $\rho$.  We often use the following equation proved by Hamilton for any gradient Ricci soliton (see for example \cite{C2009}):
\begin{equation}\label{eq:normalization}
    S + |\nabla f|^2 - 2 \rho f = C,
\end{equation}
for some constant $C$. Equation \eqref{eq:soliton} is known as the Ricci soliton equation and Equation \eqref{eq:normalization} is known as the normalization equation. The function $f$ is known as the potential function. When $\rho > 0$, $\rho = 0$, $\rho < 0$, the soliton is called shrinking, steady, expanding, respectively. For example, given an Einstein manifold $N^k$ with Ricci curvature $\operatorname{Ric}_N = \rho g_N$, consider the triple
$$(M^n = N^k \times \mathbb{R}^{n-k}, g_M = g_N + g_{\mathbb{R}^{n-k}}, f(x,y) = (\rho/2) |y|^2).$$
The triple $(M,g_M,f)$ is a gradient Ricci soliton with $\operatorname{Ric}_M + (\nabla^2f)_M=\rho g_M$. Choosing $\rho>0$ and $\rho<0$, we get an important class of Ricci shrinkers and Ricci expanders, respectively. This class of Ricci solitons, up to a quotient, are called rigid (see Petersen-Wylie \cite{PetersenWylie}). In fact, Huai-Dong Cao conjectured that Ricci shrinkers with constant scalar curvature must be rigid.

\section{Discreteness of drift Laplacians on Ricci expanders}\label{sec:discrete}

In this section, we study the discreteness of the spectra of the drift Laplacians $\Delta_{f}$ and $\Delta_{-f}$ on a Ricci expander $(M,g,f)$. We also show two operators that are unitarily equivalent to the drift Laplacian $\Delta_{f}$ on a general gradient Ricci soliton $(M,g,f)$.

By the combined works of Bakry-Émery \cite{BakryEmery}, Morgan \cite{Morgan} and Hein-Naber \cite{HeinNaber} (see Section 2 in  \cite{ChengZhou2017} for details), if a smooth metric measure space $(M,g,f)$ has Bakry-Émery-Ricci curvature bounded below by a positive constant, then the spectrum of the drift Laplacian $\Delta_f$ is discrete. In particular, the spectrum of the drift Laplacian $\Delta_f$ of a Ricci shrinker $(M,g,f)$ is discrete. We show that if the potential function $f$ of a Ricci expander $(M,g,f)$ is proper, then the spectrum of the drift Laplacian $\Delta_{f}$ is discrete. Under a simple additional condition on the scalar curvature, we show that the spectrum of the drift Laplacian $\Delta_{-f}$ is also discrete. We would like to point out that the potential function $f$ of a Ricci shrinker $(M,g,f)$ is always proper (by Cao-Zhou \cite{CaoZhou2009}).

\begin{thm}\label{thm:discrete} Let $\left(M,g,f\right)$ be a complete Ricci expander. Suppose that the potential function $f$ is proper. Then:

(i) The spectrum of the drift Laplacian $\Delta_{f}$ in the space $L_{f}^{2}$ is discrete.

(ii) Assuming that the scalar curvature $S$ is bounded above by a constant, the spectrum of the drift Laplacian $\Delta_{-f}$ in the space $L_{-f}^{2}$ is also discrete.
\end{thm}

\begin{proof}
Fix a smooth function $\phi$ on the manifold $M$. Consider the map
$$U: L^{2} \to L_{\phi}^{2}, \,\,\,\,\,\,\,\,\,\,Uu=ue^{(1/2)\phi}.$$
Here we are using the inner products
$$\langle u,v \rangle_{L^2} = \int_M u v, \,\,\,\,\,\,\,\,\,\, \langle u,v \rangle_{L^2_\phi} = \int_M u v e^{-\phi}.$$
The map $U$ is a unitary isomorphism. Consider the operator $L$ on the space $C^\infty$ given by
\[
L=\Delta+V, \,\,\,\,\,\,\,\,\,\,V=\left(1/2\right)\Delta\phi-\left(1/4\right)\left|\nabla\phi\right|^{2}.
\]
By direct calculation, we have
\[
L=U^{-1}\Delta_{\phi}U.
\]
Since the map $U$ is a unitary isomorphism, we see that the spectrum of the operator $L$ in the space $L^{2}$ is discrete if and only if the spectrum of the drift Laplacian $\Delta_{\phi}$ in the space $L_{\phi}^{2}$ is discrete. 

First we show that for $\phi=f$ the spectrum of the operator
$$L = \Delta + V, \,\,\,\,\,\,\,\,\,\,V = (1/2) \Delta f - (1/4) |\nabla f|^2,$$
in the space $L^{2}$ is discrete. Taking the trace in the Ricci soliton equation \eqref{eq:soliton}, we get the equation $\Delta f = n\rho-S$, where $n=\dim M$. Using this equation and the normalization equation \eqref{eq:normalization}, we have
$$V=\left(1/4\right)\left(2n\rho-S-C\right)-\left(1/2\right)\rho f.$$
By a result of Pigola-Rimoldi-Setti (Theorem 3 in \cite{PRS2011}),
the scalar curvature $S$ satisfies $S\geq n\rho$, so we get
\[
V\leq\left(1/4\right)\left(n\rho-C\right)-\left(1/2\right)\rho f.
\]
Using the fact that the function $\rho f$ is bounded below by a constant (this follows from the normalization equation \eqref{eq:normalization} and the above result of Pigola-Rimoldi-Setti) and also
using the fact that the potential function $f$ is proper, we have that $\rho f\to+\infty$ as $r \to+\infty$, so we find
\[
V\to-\infty\,\,\,\,\,as\,\,\,\,\,\ r \to+\infty,
\]
where $r$ is the distance function to a fixed point $x_0$. Using spectral theory (see for example page 120 in Reed-Simon's book \cite{RS1978}),
we conclude that the spectrum of the operator $L = \Delta + V$ in the space $L^{2}$ is discrete, so the spectrum of the drift Laplacian $\Delta_f$ in the space $L^2_f$ is also discrete.

Next, assuming that the scalar curvature $S$ is bounded above by a constant, we show that for $\phi=-f$ the spectrum of the operator
$$L = \Delta + V, \,\,\,\,\,\,\,\,\,\,V = - (1/2) \Delta f - (1/4) |\nabla f|^2,$$
in the space $L^{2}$ is discrete. As above, we have
$$V=\left(1/4\right)\left(-2n\rho+3S-C\right)-\left(1/2\right)\rho f.$$
As above, we have that $\rho f\to+\infty$ as $r \to+\infty$. Using this and the fact that the scalar curvature $S$ is bounded above by a constant, we find
\[
V \to-\infty\,\,\,\,\,as\,\,\,\,\,\ r \to+\infty.
\]
As above, we conclude that the spectrum of the operator $L = \Delta + V$ in the space  $L^{2}$ is discrete, so the spectrum of the drift Laplacian $\Delta_{-f}$ in the space $L^2_{-f}$ is also discrete.
\end{proof}

We can apply Theorem \ref{thm:discrete} to the product of a compact Einstein manifold with negative scalar curvature and the Gaussian expander.

\begin{ex} \label{ex:discrete}
Let $N^k$ be a compact Einstein manifold with Ricci curvature $\operatorname{Ric}_N = \rho g_N$, where $\rho<0$. The triple
$$(M^n = N^k \times \mathbb{R}^{n-k}, g_M = g_N + g_{\mathbb{R}^{n-k}}, f(x,y) = (\rho/2)|y|^2)$$
is a Ricci expander with $\operatorname{Ric}_M + (\nabla^2f)_M=\rho g_M$. The potential function $f$ is proper and the scalar curvature $S_M$ is constant. Using Theorem \ref{thm:discrete}, we conclude that the spectrum of the drift Laplacian $\Delta_f^M$ in the space $L^2_f (M)$ is discrete and the spectrum of the drift Laplacian $\Delta_{-f}^M$ in the space $L^2_{-f} (M)$ is also discrete.
\end{ex}

Applying Proposition \ref{prop:essential} to the product of the hyperbolic space and the Gaussian expander, we check that the assumption on the properness of the potential function $f$ cannot be removed from Theorem \ref{thm:discrete}.

\begin{ex}\label{ex:hyperbolic}
Let $\mathbb{H}^k$ be the hyperbolic space with Ricci curvature $\operatorname{Ric}_{\mathbb{H}^k} = \rho g_{\mathbb{H}^k}$, where $\rho<0$. The triple
$$(M^n = \mathbb{H}^k\times \mathbb{R}^{n-k}, g_M = g_{\mathbb{H}^k} + g_{\mathbb{R}^{n-k}}, f(x,y) = (\rho/2)|y|^2)$$
is a Ricci expander with $\operatorname{Ric}_M + (\nabla^2f)_M= \rho g_M$. The potential function $f$ is not proper (for example the function $f$ is constant along the set $\mathbb{H}^k\times\{0\}$), so we cannot apply Theorem \ref{thm:discrete}. On the other hand, it is well-known that the spectrum of the Laplacian $\Delta^{\mathbb{H}^k}$ of the hyperbolic space $\mathbb{H}^k$ (with the above Ricci curvature) is $\sigma( \Delta^{\mathbb{H}^k} ) = [(k-1)\rho/4,\infty)$ (see, for example, Donnelly \cite{Donnelly1981hyperbolic}), which is not discrete. Using this and Proposition \ref{prop:essential}, we see that the spectrum of the drift Laplacian $\Delta_f^M$ is not discrete.
\end{ex}

In \cite{CaoZhou2009} Cao-Zhou proved that the potential function $f$ of a Ricci shrinker $(M,g,f)$ satisfies
$$(\rho/2) (r(x)-c_1)^2 \leq f(x) \leq (\rho/2) (r(x)+c_2)^2.$$
For a Ricci expander $(M,g,f)$ with nonnegative Ricci curvature, similar estimates for the function $-f$ were proved by Cao-Catino-Chen-Mantegazza (Lemma 5.5 in \cite{Caoetal}) and Chen-Deruelle (Lemma 2.2 in \cite{ChenDer}). Using an estimate due to Munteanu-Wang \cite{MunteanuWang}, we can improve the results in \cite{Caoetal} and \cite{ChenDer} by allowing the Ricci curvature to be bounded below by a constant $\eta$, possibly nonpositive (but it cannot be positive by the Bonnet-Myers theorem). In particular, we obtain a sufficient condition for the potential function of a Ricci expander to be proper.

\begin{thm}\label{thm:asymptotic}
Let $\left(M,g,f\right)$ be a complete noncompact Ricci expander satisfying \eqref{eq:soliton} with Ricci curvature bounded below by a constant $\eta$ (that is $\operatorname{Ric\geq\eta g}$). Then there exist positive constants $c_1$ and $c_2$ such that
$$ (\eta-\rho/2)r(x)^2-c_1r(x)\leq -f(x) \leq (-\rho/2)(r(x)+c_2)^2, $$
where $r$ is the distance function.
\end{thm}

\begin{proof}
For simplicity, we write $F=-f$.

The upper bound for the function $F$ is already known and does not depend on the hypothesis of the Ricci curvature. We outline a proof here. By a result of Pigola-Rimoldi-Setti (Theorem 3 in \cite{PRS2011}), the scalar curvature $S$ satisfies $S\geq n\rho$. Using this and the normalization equation \eqref{eq:normalization}, we have
$$|\nabla F|^2\leq -2\rho F +C-n\rho.$$
Writing $\widetilde{F}=F +\frac{C-n\rho}{-2\rho}$ and dividing both sides by $\tilde{F}$ (which is nonnegative), we get
$$\left|\nabla\sqrt{\widetilde{F}}\right|\leq \sqrt{-\rho/2}.$$
This implies
$$\sqrt{\widetilde{F}(x)}\leq \sqrt{-\rho/2} r(x) + \sqrt{\widetilde{F}(x_0)}.$$
Squaring both sides and substituting $\widetilde{F}= F +\frac{C-n\rho}{-2\rho}$, we can find a positive constant $c_2$ such that
$$ F (x) \leq  (-\rho/2)(r(x)+c_2)^2.$$
    
For the lower bound, by the Ricci soliton equation \eqref{eq:soliton}, we have
$$\Ric+\nabla^2F = \Ric-\nabla^2f = 2\Ric-\rho g.$$
By the above equation and the assumption on the lower bound of the Ricci curvature, we see that the triple $(M,g,F)$ is a smooth metric measure space satisfying $\Ric_F\geq (2\eta-\rho)g$. Using this and a result of Munteanu-Wang (Theorem 0.4 in \cite{MunteanuWang}), there exists a positive constant $c_1$ such that
$$F(x)\geq (\eta-\rho/2)r(x)^2-c_1r(x).$$
\end{proof}

Using Theorem \ref{thm:discrete} and Theorem \ref{thm:asymptotic}, we show  that if a Ricci expander $(M,g,f)$ satisfies \eqref{eq:soliton} and has Ricci curvature bounded below by a constant $\eta>\rho/2$, then the spectrum of the drift Laplacian $\Delta_{f}$ is discrete. As before, under a simple additional condition on the scalar curvature, we also show that the spectrum of the drift Laplacian $\Delta_{-f}$ is discrete.

\begin{cor}\label{cor:ricci}
Let $\left(M,g,f\right)$ be a complete noncompact Ricci expander satisfying \eqref{eq:soliton} and with Ricci curvature bounded below by a constant $\eta>\rho/2$ (that is $\operatorname{Ric\geq\eta g}$). Then the potential function $f$ is proper and the spectrum of the drift Laplacian $\Delta_{f}$ in the space $L_{f}^{2}$ is discrete. Also, assuming that the scalar curvature is bounded above by a constant, the spectrum of the drift Laplacian $\Delta_{-f}$ in the space $L_{-f}^{2}$ is discrete.
\end{cor}

We can apply Corollary \ref{cor:ricci} to the Ricci expanders constructed by Bryant \cite{Bryant} and Cao \cite{Cao}.

\begin{ex} \label{ex:bryantcao}
Bryant \cite{Bryant} and Cao \cite{Cao} have constructed
nonflat rotationally symmetric expanding gradient Ricci and K\"ahler-Ricci solitons
on $\mathbb{R}^n$ and $\mathbb{C}^n$, respectively. In particular, they constructed one-parameter families
of Ricci expanders with positive sectional curvature that are asymptotic to a cone
at infinity. So, by Corollary \ref{cor:ricci}, both $\Delta_f$  and $\Delta_{-f}$ have discrete spectra on their examples. 
\end{ex}

We show two operators that are unitarily equivalent to the drift Laplacian $\Delta_{f}$ on a general Ricci soliton $(M,g,f)$. We will use this result to find the spectrum of rigid Ricci expanders in Section \ref{sec:rigid}. The result may also be of independent interest.

We will use the following useful lemma due to Cheng-Zhou (Lemma 3.1 in \cite{CZ2018}).
\begin{lemma}\label{lem:equivalent}
Let $F$ and $H$ be smooth functions on a Riemannian manifold $M$.
Then
\[
e^{-H}\Delta_{F}\left(e^{H} \cdot \right)=\Delta_{F-2H}+\Delta H+\left\langle \nabla\left(H-F\right),\nabla H\right\rangle.
\]
\end{lemma}

We can now prove the theorem.

\begin{thm}\label{thm:equivalent}
Let $\left(M^n,g,f\right)$ be a gradient Ricci soliton satisfying \eqref{eq:soliton} and with normalization \eqref{eq:normalization}. Then the following operators are unitarily equivalent:

(i) The drift Laplacian $\Delta_{f}$ in the space $L_{f}^{2}$.

(ii) The operator $\Delta+V$ in the space $L^{2}$, where $V=\left(1/4\right)\left(2n\rho-S-C\right)-\left(1/2\right)\rho f$.

(iii) The operator $\Delta_{-f}+ n \rho - S$ in the space $L_{-f}^{2}$.

In particular, if one of these operators has discrete spectrum, then all the operators have discrete spectra and
\[
\lambda_{k}\left(\Delta_{f}\right)=\lambda_{k}\left(\Delta+V\right)=\lambda_{k}\left(\Delta_{-f}-S+ n \rho \right).
\]
\
\end{thm}

\begin{proof}
First we show that the drift Laplacian $\Delta_f$ in the space $L^2_f$ is unitarily equivalent to the operator $\Delta + V$ in the space $L^2$. Consider the map
$$U: L^{2} \to L^{2}_f, \,\,\,\,\,\,\,\,\,\,Uu=ue^{\left(1/2\right)f}.$$
Here we are using the inner products
$$\langle u,v \rangle_{L^2} = \int_M u v, \,\,\,\,\,\,\,\,\,\, \langle u,v \rangle_{L^2_f} = \int_M u v e^{-f}.$$
The map $U$ is a unitary isomorphism. Taking $F=f$ and $H=(1/2)f$ in Lemma \ref{lem:equivalent}, we get
$$U^{-1} \Delta_f U = \Delta + (1/2) \Delta f - (1/4) |\nabla f|^2.$$
Taking the trace in the Ricci soliton equation \eqref{eq:soliton}, we get the equation $\Delta f = n \rho - S$. Using this equation and the normalization equation \eqref{eq:normalization}, we get
$$U^{-1} \Delta_f U = \Delta + V.$$

Next we show that the drift Laplacian $\Delta_f$ in the space $L^2_f$ is unitarily equivalent to the operator $\Delta_{-f} + n\rho - S$ in the space $L^2_{-f}$. Consider the map
$$U: L^{2}_{-f}\left(M\right) \to L^{2}_f \left(M\right), \,\,\,\,\,\,\,\,\,\, Uu=ue^f.$$
Here we are using the inner products
$$\langle u,v \rangle_{L^2_{-f}} = \int_M u v e^f, \,\,\,\,\,\,\,\,\,\, \langle u,v \rangle_{L^2_f} = \int_M u v e^{-f}.$$
The map $U$ is a unitary isomorphism. Taking $F=H=f$ in Lemma \ref{lem:equivalent}, we get
$$U^{-1} \Delta_f U = \Delta_{-f} + \Delta f.$$
Taking the trace in the Ricci soliton equation \eqref{eq:soliton}, we get the equation $\Delta f = n \rho - S$, so
$$U^{-1} \Delta_f U = \Delta_{-f} + n\rho - S.$$
\end{proof}

\begin{rmk}\label{rem:equivalent}
The explicit expression of a unitary map linking two linear operators is often useful to find eigenfunctions. In this spirit, by the proof of Theorem \ref{thm:equivalent}, for a gradient Ricci soliton $(M,g,f)$ we have
$$\Delta_f (\cdot) = e^{(1/2)f}(\Delta + V)(e^{-(1/2)f}\cdot),$$
$$\Delta_f (\cdot) = e^f (\Delta + n\rho - S)(e^{-f} \cdot).$$
\end{rmk}

\section{The spectrum of the drift Laplacian on rigid Ricci expanders}\label{sec:rigid}

In this section we find the spectrum of the drift Laplacian on a rigid Ricci expander. The method is as follows. First we find the spectrum the drift Laplacian on the Gaussian shrinker (more generally, we find the spectrum of the drift Laplacian on a rigid Ricci shrinker). Next we obtain the spectrum of the drift Laplacian on the Gaussian expander via Theorem \ref{thm:equivalent} (which gives an alternative proof of Proposition 3.1 in \cite{CZ2018}). Finally we determine the spectrum of the drift Laplacian on a rigid Ricci expander (Theorem \ref{thm:rigid}).

First we find the spectrum of the drift Laplacian on rigid Ricci shrinkers. The calculation is based on two lemmas whose proofs are standard.

\begin{lemma}\label{lemma:product}
Let $(N,g_N,f_N)$ and $(P,g_P,f_P)$ be complete smooth metric measure spaces, and consider the smooth metric measure space
$$(M = N \times P, g_M = g_N + g_P, f(x,y) = f_N(x) + f_P (y)).$$
Suppose that the spectrum of the drift Laplacian $\Delta_{f_N}^N$ in the space $L^2_{f_N} (N)$ is discrete and the spectrum of the drift Laplacian $\Delta_{f_P}^P$ in the space $L^2_{f_P} (P)$ is discrete. Then 
the spectrum of the drift Laplacian $\Delta_f^M$ in the space $L^2_f (M)$ is discrete. Also, the spectrum $\sigma (\Delta_f^M)$ is the sum of the spectra $\sigma (\Delta_{f_N}^N)$ and $\sigma (\Delta_{f_P}^P)$, that is
$$\sigma (\Delta_f^M) = \sigma (\Delta_{f_N}^N) + \sigma (\Delta_{f_P}^P).$$
In addition, any eigenfunction $u(x,y)$ of the drift Laplacian $\Delta_f^M$ is the product of an eigenfunction $v(x)$ of the drift Laplacian $\Delta_{f_N}^N$ and an eigenfunction $w(y)$ of the drift Laplacian $\Delta_{f_P}^P$, that is
$$u(x,y) = v(x) w(y).$$
\end{lemma}

\begin{proof}
Given functions $v$ and $\tilde{v}$ on the manifold $N$ and a functions $w$ and $\tilde{w}$ on the manifold $P$, consider the functions $u(x,y) = v(x) w(y)$ and $\tilde{u} (x,y) = \tilde{v} (x) \tilde{w} (y)$ on the product manifold $M = N \times P$. For such functions we have
$$\langle u,\tilde{u} \rangle_{L^2_f (M)} = \langle v, \tilde{v} \rangle_{L^2_{f_N} (N)} \langle w, \tilde{w} \rangle_{L^2_{f_P}(P)}, \,\,\,\,\,\,\,\,\,\,\Delta_f^M u = (\Delta_{f_N}^N v)w + v \Delta_{f_P}^Pw.$$
Using these formulas and standard functional analysis, we get the result.
\end{proof}

\begin{lemma}\label{lemma:spectrumR}
For the Ricci shrinker $(\mathbb{R},g_\mathbb{R},f(y) = (\rho/2)y^2)$ with $\rho>0$, the spectrum of the (one-dimensional) drift Laplacian $\Delta_f^\mathbb{R}$ in the space $L^2_f (\mathbb{R})$ is discrete, and the eigenvalues $\nu$ and corresponding eigenfunctions $w(y)$ are of the form
$$\nu = \rho p, \,\,\,\,\, w(y) = e^{-(\rho/2)y^2}H_p (\sqrt{\rho}y), \,\,\,\,\, p = 1,2,\dots,$$
where the function $H_p$ is the Hermite polynomial of degree $p$.
\end{lemma}

Applying Lemma \ref{lemma:product} to the product of an Einstein manifold with positive scalar curvature (which must be compact by the Bonnet-Myers theorem) and the Gaussian shrinker and using Lemma \ref{lemma:spectrumR}, we find the spectrum of the drift Laplacian on rigid Ricci shrinkers.

\begin{thm}\label{thm:rigid-shrinker}
Let $N^k$ be an Einstein manifold with Ricci curvature $\operatorname{Ric}_N = \rho g_N$, where $\rho>0$. Consider the triple
$$(M^{n}=N^{k}\times \mathbb{R}^{n-k}, g_M = g_N + g_{\mathbb{R}^{n-k}}, f (x,y) = \left(\rho/2\right)\left|y\right|^{2}).$$
Then:

(i) The triple $\left(M,g_M,f\right)$ is a Ricci shrinker with $\operatorname{Ric}_M + (\nabla ^2 f)_M = \rho g_M$.

(ii) The spectrum of the drift Laplacian $\Delta_{f}^M$ in the space $L_{f}^{2}\left(M\right)$ is discrete.

(iii) The eigenvalues $\lambda$ of the drift Laplacian $\Delta_{f}^M$  with corresponding eigenfunctions
$u\left(x,y\right)$  are of
the form
\[
\lambda=\mu+\rho \sum_{i=1}^{n-k}p_{i},\,\,\,\,\,u\left(x,y\right)=v\left(x\right)e^{-\left(\rho/2\right)\left|y\right|^{2}}\prod_{i=1}^{n-k}H_{p_{i}}\left(\sqrt{\rho}y_{i}\right),
\]
where the number $\mu$ is an eigenvalue of the Laplacian $\Delta^{N}$ of the Einstein manifold with corresponding eigenfunction $v\left(x\right)$, the numbers $p_1$, $\dots$, $p_{n-k}$ are nonnegative integers and the function $H_{p_{i}}$ is the Hermite polynomial of degree $p_{i}$.
\end{thm}

Next, we find the spectrum of the drift Laplacian on rigid Ricci expanders. Applying Theorem $\ref{thm:rigid-shrinker}$ to the Gaussian shrinker $\mathbb{R}^{n-k}$ and using Theorem \ref{thm:equivalent} (and Remark \ref{rem:equivalent}), we get the spectrum of the drift Laplacian of the Gaussian expander $\mathbb{R}^{n-k}$. Now applying Lemma \ref{lemma:product} to the product of a compact Einstein manifold with negative scalar curvature and the Gaussian expander, we find the spectrum of the drift Laplacian on rigid Ricci expanders. This generalizes and gives an alternative proof of a result of Cheng-Zhou for the Gaussian expander (see Proposition 3.1 in \cite{CZ2018}).

\begin{thm}\label{thm:rigid}
Let $N^k$ be a compact Einstein manifold with Ricci curvature $\operatorname{Ric}_N = \rho g_N$, where $\rho<0$. Consider the triple
$$(M^{n}=N^{k}\times \mathbb{R}^{n-k}, g_M = g_N + g_{\mathbb{R}^{n-k}}, f (x,y) = \left(\rho/2\right)\left|y\right|^{2}).$$
Then:

(i) The triple $\left(M,g_M,f\right)$ is a Ricci expander with $\operatorname{Ric}_M + (\nabla ^2 f)_M = \rho g_M$.

(ii) The spectrum of the drift Laplacian $\Delta_{f}^M$ in the space $L_{f}^{2}\left(M\right)$ is discrete.

(iii) The eigenvalues $\lambda$ of the drift Laplacian $\Delta_{f}^M$  with corresponding eigenfunctions
$u\left(x,y\right)$  are of
the form
\[
\lambda=\mu-\rho\left(n-k+\sum_{i=1}^{n-k}p_{i}\right),\,\,\,\,\,u\left(x,y\right)=v\left(x\right)e^{\left(\rho/2\right)\left|y\right|^{2}}\prod_{i=1}^{n-k}H_{p_{i}}\left(\sqrt{-\rho}y_{i}\right),
\]
where the number $\mu$ is an eigenvalue of the Laplacian $\Delta^{N}$ of the Einstein manifold with corresponding eigenfunction $v\left(x\right)$, the numbers $p_1$, $\dots$, $p_{n-k}$ are nonnegative integers and the function $H_{p_{i}}$ is the Hermite polynomial of degree $p_{i}$.
\end{thm}

\section{The second eigenvalue of the drift Laplacian on rigid Ricci expanders}\label{sec:second}

In this section, we study the first and second eigenvalues of the drift Laplacian on rigid gradient Ricci solitons in more detail. Using Theorem \ref{thm:rigid}, we see that for rigid Ricci expanders the second eigenvalue $\lambda_2(\Delta_f^M)$ of the drift Laplacian of the product manifold $M^n=N^k \times \mathbb{R}^{n-k}$ depends on the second eigenvalue (the smallest nonzero eigenvalue) $\lambda_2(\Delta^N)$ of the usual Laplacian of the Einstein factor $N$. In the following examples we show the differences between the shrinking and expanding case.

(i) (The second eigenvalue for rigid Ricci shrinkers) Let $N^k$ be a compact Einstein manifold with Ricci curvature $\operatorname{Ric}_N = \rho g_N$, where $\rho>0$. Let $M^n = N^k \times \mathbb{R}^{n-k}$ be the Ricci shrinker with the product metric and the potential function $f(x,y) = (\rho/2)|y|^2$. By Theorem \ref{thm:rigid-shrinker}, the spectrum of the drift Laplacian $\Delta_f^M$ is discrete, the first eigenvalue of this operator is
$$\lambda_1 (\Delta_f^M) = 0,$$
and the second eigenvalue of this operator (the smallest nonzero eigenvalue) is
$$\lambda_2(\Delta_f^M)=\operatorname{min}\{\lambda_2(\Delta^N),\rho\}.$$
For the Einstein manifold $N$, the smallest nonzero eigenvalue of the Laplacian $\Delta^N$ satisfies $\lambda_2(\Delta^N)\ge k\rho / (k-1) > \rho$ (by the Lichnerowicz-Obata theorem \cite{L}-\cite{O}). Thus, for $k\ge 2$ we have
$$\lambda_2(\Delta_f^M)= \rho,$$
with multiplicity $n-k$ (the dimension of the Euclidean factor), which agrees with Theorem 2 in Cheng-Zhou \cite{ChengZhou2017}.

(ii) (The second eigenvalue for rigid Ricci expanders) Let $N^k$ be a compact Einstein manifold with Ricci curvature $\operatorname{Ric}_N = \rho g_N$, where $\rho<0$. Let $M^n = N^k \times \mathbb{R}^{n-k}$ be the Ricci expander with the product metric and the potential function $f(x,y) = (\rho/2)|y|^2$. In this case the situation is quite different. For instance, zero is not an eigenvalue and the constant function is not in the space $L^2_f$. By Theorem \ref{thm:rigid}, the spectrum of the drift Laplacian $\Delta_f^M$ is discrete, the first eigenvalue of this operator is
$$\lambda_1(\Delta_f^M) = - \rho (n-k),$$
and the second eigenvalue of this operator is
$$\lambda_2(\Delta_f^M)=\operatorname{min}\{ -\rho(n-k+1),  \lambda_2(\Delta^N)-\rho(n-k)\}.$$
The value and the multiplicity of the eigenvalue $\lambda_2(\Delta_f^M)$ depends on whether the eigenvalue $\lambda_2(\Delta^N)$ is greater than, equal to or less than $-\rho$. If $\lambda_2(\Delta^N)>-\rho$, then $\lambda_2(\Delta_f^M)$ is equal to $-\rho(n-k)$ and its multiplicity is $n-k$. If $\lambda_2(\Delta^N)=-\rho$, then $\lambda_2(\Delta_f^M)$ is equal to $-\rho(n-k)$ and its multiplicity is the sum of the multiplicity of $\lambda_2(\Delta^N)$ and $n-k$. If $\lambda_2(\Delta^N)<-\rho$, then $\lambda_2(\Delta_f^M)$ is equal to $\lambda_2 (\Delta^\Sigma) -\rho(n-k)$ and its multiplicity is the multiplicity of $\lambda_2(\Delta^N)$.

Thus, on rigid Ricci shrinkers the first eigenvalue of the drift Laplacian is equal to zero, and, on the other hand, on rigid Ricci expanders the first eigenvalue of the drift Laplacian is nonzero and equal to $- \rho(n-k)$ (of course the multiplicity is $1$ in both cases). Also, on rigid Ricci shrinkers the second eigenvalue of the drift Laplacian is equal to the constant of Ricci the soliton equation with multiplicity equal to the dimension of Euclidean factor, and, on the other hand, on rigid Ricci expanders the second eigenvalue of the drift Laplacian and its multiplicity are both less clear, so a splitting result in the spirit of Cheng-Zhou \cite{ChengZhou2017} may not be true. This led us to investigate the second eigenvalue of the drift Laplacian on the simplest non-trivial example of a rigid Ricci expander we could think of, namely the product of a closed hyperbolic surface with the Euclidean space.

Let $\Sigma$ be a closed Riemann surface with constant curvature $\rho<0$. We attempt to estimate the smallest nonzero eigenvalue of the surface $\Sigma$. The classical Yang-Yau inequality \cite{YangYau} states that the smallest nonzero eigenvalue $\lambda_2(\Delta^\Sigma)$ of the Laplacian on a closed Riemann surface $\Sigma$ of genus $\gamma$ satisfies
$$\lambda_2(\Delta^\Sigma)\operatorname{Area(\Sigma)}\leq 8\pi\left\lfloor \frac{\gamma+3}{2} \right\rfloor.$$
Using the Gauss-Bonnet theorem, we see that $\rho \operatorname{Area}(\Sigma) = 4\pi(1 -\gamma),$
so the Yang-Yau inequality becomes
$$\lambda_2(\Delta^\Sigma)\leq \frac{ 2\rho\left\lfloor \frac{\gamma+3}{2} \right\rfloor}{1-\gamma}.$$
We would like to show that $\lambda_2(\Delta^\Sigma)\leq-\rho$, but the right hand side of the above inequality is never less than or equal to $-\rho$, so the Yang-Yau inequality is not sufficient for our purposes. Very recently Karpukhin-Vinokurov \cite{KarpukhinVinokurov} obtained the following improvement of the Yang-Yau inequality:
\begin{equation*}
    \lambda_2(\Delta^\Sigma)\operatorname{Area(\Sigma)}\leq \frac{2\pi}{13-\sqrt{15}}\left( \gamma+\left(33-4\sqrt{15}\right)\left\lceil 
\frac{5\gamma}{6} \right\rceil+4\left( 41-5\sqrt{15} \right) \right).
\end{equation*}
Using again the Gauss-Bonnet theorem, the above inequality becomes
$$\lambda_2(\Delta^\Sigma)\leq \frac{\rho}{2(13-\sqrt{15})(1-\gamma)}\left( \gamma+\left(33-4\sqrt{15}\right)\left\lceil 
\frac{5\gamma}{6} \right\rceil+4\left( 41-5\sqrt{15} \right) \right).$$
Using the fact that $ \left\lceil 
\frac{5\gamma}{6} \right\rceil\leq  \frac{5\gamma}{6}+1$, we see that the right hand side of the above inequality is less than $-\rho$ for $\gamma \geq 46$.

We summarize the above discussion in the following result:
\begin{thm}\label{thm:second}
Let $\Sigma$ be a closed Riemann surface with constant curvature $\rho<0$. Let $M^n=\Sigma\times\mathbb{R}^{n-2}$ be the Ricci expander with the product metric and the potential function $f(x,y)=(\rho/2)|y|^2$. Then the spectrum of the drift Laplacian $\Delta_f^M$ is discrete, the first eigenvalue is
$$\lambda_1(\Delta_f^M)=-\rho(n-k),\,\,with\,\,multiplicity\,\,1,$$
and the second eigenvalue is
$$
\lambda_2 (\Delta_f^M) = 
\begin{cases}
-\rho (n-2),\,\,with\,\,multiplicity\,\,n-2, & if\,\,\lambda_2 (\Delta^\Sigma )>-\rho,\\
-\rho (n-2),\,\,with\,\,the\,\,multiplicity\,\,of\,\,\lambda_2 (\Delta^\Sigma)+n-2, & if\,\,\lambda_2 ( \Delta^\Sigma )=-\rho,\\
\lambda_2 ( \Delta^\Sigma) -\rho (n-2),\,\,with\,\,the\,\,multiplicity\,\,of\,\,\lambda_2 ( \Delta^\Sigma ), & if\,\,\lambda_2 ( \Delta^\Sigma )<-\rho.
\end{cases}
$$
Also, if the genus $\gamma\geq 46$, then $\lambda_2 ( \Delta^\Sigma )<-\rho$ and
$$\lambda_2(\Delta_f^M)=\lambda_2(\Delta^\Sigma)-\rho(n-2).$$
\end{thm}

In \cite{SU2018} Strohmaier-Uski showed numerically that for  the Bolza surface $\Sigma$, which has constant negative curvature $\rho=-1$ and genus $\gamma = 2$, the smallest nonzero eigenvalue of the Laplacian satisfies $\lambda_2(\Delta^{\Sigma})\approx 3.8$ (see section 5.3 in \cite{SU2018}). In particular, for the Bolza surface $\Sigma$ we have $\lambda_2(\Delta^{\Sigma}) > -\rho$. Taking $\Sigma$ as the Bolza surface in Theorem \ref{thm:second}, we see that the second eigenvalue of the drift Laplacian of the rigid Ricci expander $M^n=\Sigma\times\mathbb{R}^{n-2}$ satisfies $\lambda_2(\Delta_f^M)=-\rho(n-2)$. It is a known conjecture that the Bolza surface has the largest eigenvalue among all closed hyperbolic surfaces of genus $2$. The other closed hyperbolic surfaces of genus $2$ may have a smaller eigenvalue and the inequality $\lambda_2(\Delta^{\Sigma}) > -\rho$ may not be valid, but at least we know that $\lambda_2(\Delta^{\Sigma}) < -\rho$  when the genus $\gamma \geq 46$. 

\section{Appendix: Spectrum on product manifolds}\label{sec:product}
In this section we study, for the sake of completeness, some spectral properties of the product of two smooth metric measure spaces.
\begin{prop} 
    Let $M_1$ and $M_2$ be complete Riemannian manifolds and let $f_i$ be a smooth function on $M_i$ ($i=1,2$). Consider the smooth metric measure space $(M=M_1 \times M_2,g,f)$ with the product metric $g$ and the potential function $f(x,y)=f_1(x)+f_2(y)$. Then
    $$\bottom(\Delta_f^M)=\bottom(\Delta_{f_1}^{M_1})+\bottom(\Delta_{f_2}^{M_2}),$$
    where $\bottom(\Delta_f^M)$ and $\bottom(\Delta_{f_i}^{M_i})$ denotes the bottom spectrum of the drift Laplacian on $M$ and $M_i$, respectively.
\end{prop}

\begin{proof}
Consider the closed ball $B_r^M=B_r^M(p_1,p_2)$ of radius $r$ and center $(p_1,p_2)\in M$ and the product $U_r=B_r^{M_1} \times B_r^{M_2}$, where $B_r^{M_i} = B_r^{M_i}(p_i)$. Using the fact that $ B_{r}^M\subset U_r\subset B_{2r}^M$ and the variational characterization of the the first eigenvalue of the drift Laplacian with Dirichlet boundary conditions (on compact domains with boundary), we have
$$\lambda^D_1(\Delta_f^{B_{2r}^M})\leq\lambda^D_1(\Delta_f^{U_r})\leq \lambda^D_1(\Delta_f^{B_{r}^M}).$$
On the other hand, for the product $U_r=B_r^{M_1}\times B_r^{M_2}$ it is known that
$$\lambda^D_1(\Delta_f^{U_r})=\lambda^D_1(\Delta_{f_1}^{B_r^{M_1}})+\lambda^D_1(\Delta_{f_2}^{B_r^{M_2}}).$$
Thus
$$\lambda^D_1(\Delta_f^{B_{2r}^M})\leq\lambda^D_1(\Delta_{f_1}^{B_r^{M_1}})+\lambda^D_1(\Delta_{f_2}^{B_r^{M_2}})\leq \lambda^D_1(\Delta_f^{B_{r}^M}).$$
Using the fact that the bottom spectrum of the drift Laplacian on a complete manifold is the limit (as the radius tends to infinity) of the first eigenvalue of the drift Laplacian with Dirichlet boundary condition on an exhausting sequence of balls, and taking the limit as $r\to \infty$, we get the result.
\end{proof}

We will use the following well-known lemma due to Donnelly (see \cite{Donnelly1981essential}, Proposition 2.2).

\begin{lemma}[Weyl's criterion]\label{lem:Weyl}
    Let $H$ be a Hilbert space and $A:H\rightarrow H$ be a self-adjoint operator. A necessarry and sufficient condition for the interval $(\lambda-\epsilon,\lambda+\epsilon)$ to intersect the essential spectrum of $A$ is that there exists an infinite dimensional subspace $G$ of $H$ for which $$||(A+\lambda)u||_H<\epsilon ||u||_H,\ \forall u\in G.$$
\end{lemma}

Denote by $b+A=\{b+a;a\in A\}$ the translation of a nonempty set $ A\subset \mathbb{R}$ by $b\in\mathbb{R}$.

\begin{prop}
    Let $M=N \times P$ be a product manifold (with the product metric). Suppose that $\sigma_{ess}(\Delta^N)\neq\emptyset$ and $\sigma(\Delta^P)=\sigma_d(\Delta^P)$. Then
    $$\sigma_{ess}(\Delta^M)\supset\bigcup_{\mu\in\sigma_d(\Delta^N)}(\mu+\sigma_{ess}(\Delta^N)).$$
\end{prop}

\begin{proof}
Let $\mu\in\sigma_d(\Delta^P)$ be an eigenvalue of the Laplacian $\Delta^P$ with eigenfunction $v\in L^2(P)$, that is,
$$\Delta^P v+\mu v=0.$$
We can assume that $ ||v||_{L^2(P)}=1$. Let $\lambda\in\sigma_{ess}(\Delta^N)$. By Lemma \ref{lem:Weyl}, given $\epsilon>0$ there exists an infinite dimensional subspace $G_\epsilon$ of $L^2(N)$ such that
$$ ||(\Delta^N +\lambda )u||_{L^2(N)}<\epsilon ||u||_{L^2(N)},\ \forall u\in G_\epsilon. $$
Let $\{u_i;i\in\mathbb{N}\}$ be an orthonormal basis for $G_\epsilon$ and consider the functions $\phi_i(x,y)=u_i(x)v(y)$. A direct calculation shows that $(\phi_i,\phi_j)_{L^2(M)}=\delta_{ij}$, where $\delta_{ij}$ is the Kronecker delta, which implies that $\operatorname{span}\{\phi_i;i\in\mathbb{N}\}$ is an infinite dimensional subspace of $L^2(M)$. Moreover,
$$||\Delta^M \phi_i +(\lambda+\mu)\phi_i||_{L^2(M)}=||\Delta^N u_i+\lambda u_i||_{L^2(N)}||v||_{L^2(P)}<\epsilon||u_i||_{L^2(N)}=\epsilon||\phi_i||_{L^2(M)}.$$
Using again Lemma \ref{lem:Weyl}, we see that $\sigma_{ess}(\Delta^M)\cap(\lambda+\mu-\epsilon,\lambda+\mu+\epsilon)\neq\emptyset$. Since $\epsilon$ is arbitrary, we conclude that $\lambda+\mu\in\sigma_{ess}(\Delta^M)$.
\end{proof}

The next proposition, which is used in Example \ref{ex:hyperbolic}, has exactly the same proof of the previous proposition. Recall that for the Euclidean space $\mathbb{R}^{n-k}$ and $f(y)=(\rho / 2) |y|^2$ with $\rho<0$ the spectrum of the drift Laplacian $\Delta_f^{\mathbb{R}^{n-k}}$ is discrete.

\begin{prop}\label{prop:essential} Let $M^n=N^k\times \mathbb{R}^{n-k}$ be a product manifold (with the product metric) and let $f(x,y)=(\rho / 2) |y|^2$, where $\rho<0$. Suppose that $\sigma_{ess}(\Delta^N)\neq\emptyset$. Then
    $$\sigma_{ess}(\Delta_f^M)\supset\bigcup_{\mu\in\sigma(\Delta_f^{\mathbb{R}^{n-k}})}(\mu+\sigma_{ess}(\Delta^N)).$$
\end{prop}

\end{document}